\documentclass[11pt]{amsart}
\usepackage{amsmath,amssymb,amsthm}
\theoremstyle{plain}
\newtheorem{thm}{Theorem}[section]
\newtheorem{lem}[thm]{Lemma}

\newtheorem{defn}[thm]{Definition}
\newtheorem{fact}[thm]{Fact}
\newtheorem{example}[thm]{Example}

\newtheorem{rem}[thm]{Remark}

\newcommand{\bijarrow}{\mathrel{\stackrel{1:1}{\longleftrightarrow}}}

\newcommand{\Ad}{\mathop{\mathrm{Ad}}\nolimits}

\newcommand{\Hom}{\mathop{\mathrm{Hom}}\nolimits}

\newcommand{\hyp}{\text{hyp}}

\newcommand{\R}{\mathbb{R}}

\newcommand{\diag}{\mathop{\mathrm{diag}}\nolimits}
\newcounter{mycounter}

\usepackage{hyperref}
\begin{document}

\title[Non virtually abelian discontinuous groups for $G/H$]
{Homogeneous space with non virtually abelian discontinuous groups but without any proper $SL(2,\mathbb{R})$-action}
\author{Takayuki Okuda}
\address{Department of Mathematics, Graduate School of Science, Hiroshima University 1-3-1 Kagamiyama, Higashi-Hiroshima, 739-8526 Japan}
\email{okudatak@hiroshima-u.ac.jp}
\date{}
\maketitle

\begin{abstract}
In the study of discontinuous groups for non-Riemannian homogeneous spaces,
the idea of 
``continuous analogue'' gives a powerful method (T.~Kobayashi [Math.~Ann.~1989]).
For example, a semisimple symmetric space $G/H$ admits 
a discontinuous group which is not virtually abelian
if and only if $G/H$ admits a proper $SL(2,\mathbb{R})$-action
(T.~Okuda [J.~Differential Geom.~2013]).
However, the action of discrete subgroups is not always approximated by
that of connected groups.
In this paper, we show that the theorem cannot be 
extended to general homogeneous spaces $G/H$ of reductive type.
We give a counterexample in the case $G = SL(5,\R)$.
\end{abstract}


\section{Introduction and statement of main results.}

It is notorious in pseudo-Riemannian geometry that
the action of isometric discrete groups is not always properly 
discontinuous.
A typical example is a homogeneous space $G/H$ with 
$G$, $H$ real reductive Lie groups
where any discrete subgroup $\Gamma$ of $G$ 
acts isometry on $G/H$ endowed with a pseudo-Riemannian metric
by the Killing form, 
but $\Gamma$ may not be a {\textit{discontinuous 
group for}} $G/H$
i.e.~$\Gamma$ may not act properly discontinuously on $G/H$.

By T.~Kobayashi \cite{Kobayashi89}, 
a homogeneous space $G/H$ of reductive type 
admits an infinite discontinuous group if and only if $G/H$ admits a proper action of a one-dimensional non-compact closed subgroup of $G$.
This idea of ``continuous analogue'' was the key to 
the proof \cite{Kobayashi89} of a necessary and sufficient condition for the Calabi--Markus phenomenon \cite{Calabi-Murkus62}.

In \cite[Theorem 1.3]{Okuda13cls},
by using Benoist's criterion \cite{Benoist96} and structure theory of semisimple symmetric spaces, 
we proved that a semisimple symmetric space $G/H$ admits a non virtually abelian properly discontinuous group if and only if $G/H$ admits a proper action of a three-dimensional simple subgroup $L$ of $G$ (i.e.~$L$ is either $SL(2,\mathbb{R})$ or $PSL(2,\R)$).
In general, the latter condition (``continuous analogue'') implies the former because $SL(2,\mathbb{R})$ contains non virtually abelian discrete subgroups.
However, it was left open whether or not the former condition (``discrete group'') implies the latter without the assumption that $G/H$ is a semisimple symmetric space.

Although important criteria for proper actions (\cite{Benoist96,Kobayashi89,Kobayashi96}) 
do not require that $G/H$ is a symmetric space,
we find that the expected equivalence does not hold for general non-symmetric spaces.
Here is our main theorem:

\begin{thm}\label{thm:main}
There exists a $3$-dimensional split abelian subgroup $H$ of $SL(5,\mathbb{R})$ satisfying the following$:$
\begin{enumerate}
\item \label{item:Main:nv} There exists a discrete subgroup $\Gamma$ of $SL(5,\mathbb{R})$ such that $\Gamma$ is not virtually abelian and the $\Gamma$-action on the homogeneous space $SL(5,\mathbb{R})/H$ is properly discontinuous.
\item For any Lie group homomorphism $\Phi: SL(2,\mathbb{R}) \rightarrow SL(5,\mathbb{R})$, the action of $SL(2,\mathbb{R})$ on the homogeneous space $SL(5,\mathbb{R})/ H$ via $\Phi$ is not proper.
\end{enumerate}
\end{thm}

Actually, 
we can construct $\Gamma$ in \eqref{item:Main:nv} of Theorem \ref{thm:main}
as a Zariski dense free subgroup in $SL(5,\mathbb{R})$,
by \cite[Section 7]{Benoist96}.

\section{Criterion of proper actions}\label{section:preliminary}

In this section, we recall results of T.~Kobayashi \cite{Kobayashi89} and Y.~Benoist \cite{Benoist96} in a way that we shall need.

Let $G$ be a linear reductive Lie group, namely, $G$ is a real form of a connected complex reductive Lie group $G_\mathbb{C}$, and $H$ a reductive subgroup of $G$.
We denote by $\mathfrak{g}$ and $\mathfrak{h}$ the Lie algebras of $G$ and $H$, respectively.
Let us take a maximally split abelian subspace $\mathfrak{a}$ of $\mathfrak{g}$.
Let $W$ denote the Weyl group for the restricted
root system $\Sigma$ of $(\mathfrak{g},\mathfrak{a})$.
Then any maximally split abelian subspace of $\mathfrak{h}$ can be transformed into a subspace of $\mathfrak{a}$ by an inner-automorphism of $\mathfrak{g}$.
We denote this subspace by $\mathfrak{a}_\mathfrak{h}$, which is uniquely determined up to the Weyl group $W$.
An analogous notation will be applied to another reductive subgroup $L$ of $G$.

We fix a simple system $\Pi$ of the restricted root system $\Sigma$ of $(\mathfrak{g},\mathfrak{a})$. 
We write $\mathfrak{a}_+$ for the closure of the dominant Weyl chamber,
and $w_0$ for the longest element in $W$ corresponding to the simple system $\Pi$.
Then the linear transformation $x \mapsto -w_0 \cdot x$ on $\mathfrak{a}$ leaves the closed Weyl chamber $\mathfrak{a}_+$ invariant.
We put 
\begin{align*}
\mathfrak{b}_+ := \{\, A \in \mathfrak{a}_+ \mid -w_0 \cdot A = A \,\}.
\end{align*}

\begin{example}\label{ex:aa+b+_for_SL5}
Let $G = SL(5,\mathbb{R})$,
we can take $\mathfrak{a}$, $\mathfrak{a}_+$ and $\mathfrak{b}_+$ as 
\begin{align*}
\mathfrak{a} &= \left\{  \diag(a_1,\dots,a_5)~\middle|~a_1,\dots,a_5 \in \mathbb{R},\ \sum_{i=1}^{5} a_i =0 \right\}, \\
\mathfrak{a}_+ &= \left\{ \diag(a_1,\dots,a_5) \in \mathfrak{a} \mid a_1 \geq a_2 \geq a_3 \geq a_4 \geq a_5  \right\}, \\
\mathfrak{b}_+ &= \left\{ \diag(b_1,b_2,0,-b_2,-b_1) \in \mathfrak{a}_+ \mid b_1,b_2 \in \mathbb{R},~b_1 \geq b_2 \right\},
\end{align*}
where the Weyl group $W$ is isomorphic to the symmetric group $\mathfrak{S}_5$ and acts on $\mathfrak{a}$ as permutations of the coordinates $a_1,\dots,a_5$.
\end{example}

A continuous action of a locally compact group $G$ on a locally compact topological space $X$ is called proper if $\{\, g \in G \mid gS \cap S \neq \emptyset \,\}$ is compact for any compact subset $S$ of $X$.
It is properly discontinuous if $L$ is discrete.

The next criterion is proved by T.~Kobayashi in \cite{Kobayashi89}:
\begin{fact}[{\cite[Theorem 4.1]{Kobayashi89}}]\label{fact:Kobayashi89}
Let $H$, $L$ be reductive subgroups of a linear reductive Lie group $G$.
Then the following two conditions on $(G,H,L)$ are equivalent$:$
\begin{enumerate}
\item The $L$-action on $G/H$ is proper.
\item $\mathfrak{a}_\mathfrak{h} \cap W \cdot \mathfrak{a}_\mathfrak{l} \neq \{0\}$.
\end{enumerate}
\end{fact}

Let $\phi$ be a Lie algebra homomorphism $\phi : \mathfrak{sl}(2,\mathbb{R}) \rightarrow \mathfrak{g}$. 
We denote by $\mathcal{O}^{\text{hyp}}_{\phi}$ the adjoint orbit through $\phi \begin{pmatrix} 1 & 0 \\ 0 & -1 \end{pmatrix}$ in $\mathfrak{g}$.
Then it is known that $\mathcal{O}^{\hyp}_{\phi}$ and $\mathfrak{a}_+$ meet at one point.

\begin{defn}\label{defn:Aphi}
For a Lie algebra homomorphism $\phi : \mathfrak{sl}(2,\mathbb{R}) \rightarrow \mathfrak{g}$,
we denote by $A_\phi$ the unique element in $\mathfrak{a}_+ \cap \mathcal{O}^{\text{hyp}}_\phi$. 
\end{defn}

The next fact for the proper actions of $SL(2,\mathbb{R})$ follows from Fact \ref{fact:Kobayashi89}:

\begin{fact}[Corollary to Fact \ref{fact:Kobayashi89}]\label{fact:Kobayashi_sl2}Let $G$ be a linear reductive Lie group, $H$ a reductive subgroup of $G$, 
and $\Phi : SL(2,\mathbb{R}) \rightarrow G$ a Lie group homomorphism.
We denote by $\phi : \mathfrak{sl}(2,\mathbb{R}) \rightarrow \mathfrak{g}$ 
the differential of $\Phi$,
and take the element $A_\phi$ in $\mathfrak{a}_+$ described above.
Then the following conditions on $\Phi$ are equivalent$:$
\begin{enumerate}
\item The $SL(2,\mathbb{R})$-action on $G/H$ via $\Phi$ is proper.
\item The element $A_{\phi}$ is not in $W \cdot \mathfrak{a}_\mathfrak{h}$.
\end{enumerate}
\end{fact}

Let us denote by $\Hom(\mathfrak{sl}(2,\mathbb{R}),\mathfrak{g})$ 
the set of all Lie algebra homomorphisms 
from $\mathfrak{sl}(2,\mathbb{R})$ to $\mathfrak{g}$.
By Fact \ref{fact:Kobayashi_sl2}, 
if the subset $\{\, A_\phi \mid \phi \in \Hom(\mathfrak{sl}(2,\mathbb{R}),\mathfrak{g}) \,\}$ 
of $\mathfrak{a}_+$ is contained in $W \cdot \mathfrak{a}_\mathfrak{h}$,
then for any Lie group homomorphism $\Phi : SL(2,\mathbb{R}) \rightarrow G$, 
the $SL(2,\mathbb{R})$-action on $G/H$ via $\Phi$ is not proper.

\begin{rem}
Let us define an equivalent relation on $\Hom(\mathfrak{sl}(2,\mathbb{R}),\mathfrak{g})$ by 
$\phi \sim \phi'$ if there exists $l \in SL(2,\mathbb{R})$ and $g \in G$ such that 
$\phi' = \Ad_\mathfrak{g}(g) \circ \phi \circ \Ad_{\mathfrak{sl}(2,\mathbb{R})}(l)$.
Then we have a natural surjection 
from $\Hom(\mathfrak{sl}(2,\mathbb{R}),\mathfrak{g})/{\sim}$ to the set  
$\{\, A_\phi \mid \phi \in \Hom(\mathfrak{sl}(2,\mathbb{R}), \mathfrak{g}) \,\}$.
We remark that the map may not be injective 
$($See \cite[\text{Chapter }9]{Collingwood-McGovern93} for more details$)$. 
However, by combining \cite[Proposition 4.5 $(iii)$]{Okuda13cls} with 
the Jacobson--Morozov theorem and results of Kostant \cite{Kostant59} and Malcev \cite{Malcev50}, 
we see that there exists a bijection below$:$
\begin{align*}
\{\, A_\phi \mid \phi &\in \Hom(\mathfrak{sl}(2,\mathbb{R}), \mathfrak{g}) \,\} \\
&\bijarrow \{ \text{Complex nilpotent orbits $\mathcal{O}$ in $\mathfrak{g}_\mathbb{C}$ such that $\mathcal{O} \cap \mathfrak{g} \neq \emptyset$} \}.
\end{align*}
\end{rem}

The next fact for the existence of properly discontinuous actions of a non virtually abelian discrete group is proved by Y.~Benoist in \cite{Benoist96}:

\begin{fact}[{\cite[Theorem 1.1]{Benoist96}}]\label{fact:Benoist}
Let $G$ be a linear reductive Lie group, $H$ a reductive subgroup of $G$.
The following conditions on $(G,H)$ are equivalent$:$
\begin{enumerate}
\item There exists a non virtually abelian discrete subgroup $\Gamma$ of $G$ such that the $\Gamma$-action on $G/H$ is properly discontinuous.
\item $\mathfrak{b}_+ \not \subset W \cdot \mathfrak{a}_\mathfrak{h}$.
\end{enumerate}
\end{fact}

\section{An example of $SL(5,\mathbb{R})$-spaces}\label{section:prove_main}

Let $G = SL(5,\mathbb{R})$ and take $\mathfrak{a}$, $\mathfrak{a}_+$ and $\mathfrak{b}_+$ as in Example \ref{ex:aa+b+_for_SL5}.
We write $\langle \ ,\ \rangle$ for the inner product on $\mathfrak{a}$ induced by the Killing form on $\mathfrak{g}$.
That is, we put
\[
\langle \diag(a_1,\dots,a_5), \diag(a'_1,\dots,a'_5) \rangle := \sum_{i=1}^{5} a_i \cdot a'_i \in \mathbb{R}
\]
for each $\diag(a_1,\dots,a_5), \diag(a'_1,\dots,a'_5) \in \mathfrak{a}$.

We define a $3$-dimensional abelian subalgebra $\mathfrak{h}$ by
\[
\mathfrak{h} := \{\, \diag(a_1,\dots,a_5) \in \mathfrak{a} \mid 6a_1+6a_2+a_3-4a_4-9a_5 = 0 \,\} \subset \mathfrak{a}.
\]
In other wards, $\mathfrak{h}$ is the orthogonal complement subspace of $\mathfrak{a}$ for $\diag(6,6,1,-4,-9)$.
Then
\[
H := \exp \mathfrak{h} \subset SL(5,\mathbb{R})
\]
is a split abelian subgroup of $SL(5,\mathbb{R})$ with $H \simeq \mathbb{R}^3$.
In this case, we take $\mathfrak{a}_\mathfrak{h}$ as $\mathfrak{h}$ itself (see for the notation in Section \ref{section:preliminary}).
By using Fact \ref{fact:Kobayashi_sl2} and Fact \ref{fact:Benoist},
we can reduce the proof of Theorem \ref{thm:main} to the following two lemmas:

\begin{lem}\label{lem:prove:b_+}
$\mathfrak{b}_+ \not \subset W \cdot \mathfrak{a}_\mathfrak{h}$.
\end{lem}

\begin{lem} \label{lem:prove:sl2} 
For any Lie algebra homomorphism $\phi : \mathfrak{sl}(2,\mathbb{R}) \rightarrow \mathfrak{sl}(5,\mathbb{R})$, 
the element $A_\phi$ $($see Definition \ref{defn:Aphi} for the notation$)$ is in $W \cdot \mathfrak{a}_\mathfrak{h}$.
\end{lem}

First, we shall show Lemma \ref{lem:prove:b_+} as follows:
\begin{proof}[Proof of Lemma \ref{lem:prove:b_+}]
Let us take $r_1,r_2 \in \mathbb{R}$ with $r_1 > r_2 > 0$ such that $r_1$ and $r_2$ are linearly independent over $\mathbb{Z}$.
For example, we can take $(r_1,r_2) = (\sqrt{2},1)$.
We shall prove that the element 
$
\diag(r_1,r_2,0,-r_2,-r_1)
$
of $\mathfrak{b}_+$ is not in $W \cdot \mathfrak{a}_\mathfrak{h}$.
If 
\[
\langle \diag(r_1,r_2,0,-r_2,-r_1), \diag(a_1,a_2,a_3,a_4,a_5) \rangle = 0
\]
for some $a_1,\dots,a_5 \in \mathbb{Z}$,
then we have $a_1 = a_5$ and $a_2 = a_4$ since $r_1$ and $r_2$ are linearly independent over $\mathbb{Z}$.
Hence 
\[
\langle \diag(r_1,r_2,0,-r_2,-r_1), \sigma \diag(6,6,1,-4,-9) \rangle \neq 0
\]
for any $\sigma \in \mathfrak{S}_5 = W$. 
Therefore, we obtain that \[
\diag(r_1,r_2,0,-r_2,-r_1) \not \in W \cdot \mathfrak{a}_\mathfrak{h}.
\]
\end{proof}

To describe the proof of Lemma \ref{lem:prove:sl2},
we use the next fact for the set $\{\, A_\phi \mid \phi \in \Hom(\mathfrak{sl}(2,\mathbb{R}),\mathfrak{sl}(5,\mathbb{R})) \,\}$,
where $\Hom(\mathfrak{sl}(2,\mathbb{R}),\mathfrak{sl}(5,\mathbb{R}))$ is the set of all Lie algebra homomorphisms 
from $\mathfrak{sl}(2,\mathbb{R})$ to $\mathfrak{sl}(5,\mathbb{R})$:

\begin{fact}\label{fact:table_sl2_sl5}
The set $\{\, A_\phi \mid \phi \in \Hom(\mathfrak{sl}(2,\mathbb{R}),\mathfrak{sl}(5,\mathbb{R})) \,\}$ $($see Definition $\ref{defn:Aphi}$ for the notation$)$ is parametrised by partitions of $5$ as in Table \ref{table:sl5}.
\begin{table}
\caption{The list of elements in $\{\, A_\phi \mid \phi \in \Hom(\mathfrak{sl}(2,\mathbb{R}),\mathfrak{sl}(5,\mathbb{R})) \,\}$}\label{table:sl5}
\begin{center}
\begin{tabular}{cc} \hline
Partition of $5$ & $A_\phi$ \\ \hline
$[5]$ & $\diag(4,2,0,-2,-4)$ \\
$[4,1]$ & $\diag(3,1,0,-1,-3)$ \\
$[3,2]$ & $\diag(2,1,0,-1,-2)$ \\
$[3,1^2]$ & $\diag(2,0,0,0,-2)$ \\
$[2^2,1]$ & $\diag(1,1,0,-1,-1)$ \\
$[2,1^3]$ & $\diag(1,0,0,0,-1)$ \\
$[1^5]$ & $\diag(0,0,0,0,0)$ \\ \hline
\end{tabular}
\end{center}
\end{table}
\end{fact}

Fact \ref{fact:table_sl2_sl5} can be obtained by combining
\cite[Chapter 3.6]{Collingwood-McGovern93} with 
{\cite[Theorem 9.3.3 and remarks after Theorem 9.3.5]{Collingwood-McGovern93}}.

\begin{proof}[Proof of Lemma \ref{lem:prove:sl2}]
By Fact \ref{fact:table_sl2_sl5}, 
we can observe that $A_\phi$ 
can be written by a scalar multiple of 
$\diag(3,1,0,-1,-3)$, $\diag(2,1,0,-1,-2)$, $\diag(1,1,0,-1,-1)$ or $\diag(1,0,0,0,-1)$ 
for any $\phi \in \Hom(\mathfrak{sl}(2,\mathbb{R}),\mathfrak{sl}(5,\mathbb{R}))$. 
Here, we compute the inner-products that 
\begin{align*}
\langle \diag(3,1,0,-1,-3), \diag(6,-9,-4,6,1) \rangle &= 0,\\
\langle \diag(2,1,0,-1,-2), \diag(6,-4,-9,6,1) \rangle &= 0,\\
\langle \diag(1,1,0,-1,-1), \diag(6,-9,6,-4,1) \rangle &= 0,\\
\langle \diag(1,0,0,0,-1), \diag(6,-9,-4,1,6)  \rangle &= 0.
\end{align*}
This means that these are in $W \cdot \mathfrak{a}_\mathfrak{h}$.
This completes the proof.
\end{proof}

Then we complete the proof of Theorem \ref{thm:main}.

\begin{rem}
For semisimple symmetric pairs $(\mathfrak{g},\mathfrak{h})$, 
the subset $\mathfrak{a}_+ \cap \Ad(G) (\mathfrak{p} \cap \mathfrak{h})$ of $\mathfrak{a}$ is a convex cone $($see \cite[Lemma 4.9]{Okuda13cls} and its proof$)$. 
This was one of key lemmas for the
proof of the main results for symmetric spaces in \cite{Okuda13cls},
however, owing to Theorem \ref{thm:main}, 
we see that an analogous statement to
\cite[Lemma 4.9]{Okuda13cls} does not hold 
for the above non-symmetric pair $(G,H)$.
\end{rem}

{\bf Final remark.}
For a semisimple symmetric space $G/H$,
it was also proved in \cite{Okuda13cls} that 
$G/H$ admits a discontinuous group $\Gamma$ which is 
isomorphic to a surface group 
(i.e.~the fundamental group of a closed Riemann surface of genus $g \geq 2$)
if and only if $G/H$ admits a proper action of $SL(2,\mathbb{R})$.
I do not know if such a statement holds in general when $G/H$ is non-symmetric.

\section*{Acknowledgements.}
To discover the example, 
the author owes much to Yves Benoist for helpful discussions
during his visit to Tokyo in 2012.
The main result of this paper was announced 
in Symposium on Representation Theory
organized by Hiroyuki Ochiai and Minoru Ito at Kagoshima on December 2012.
The author also would like to give warm thanks to Toshiyuki Kobayashi
whose comments were of inestimable value for this paper.

\def\cprime{$'$}
\providecommand{\bysame}{\leavevmode\hbox to3em{\hrulefill}\thinspace}
\providecommand{\MR}{\relax\ifhmode\unskip\space\fi MR }
\providecommand{\MRhref}[2]{%
  \href{http://www.ams.org/mathscinet-getitem?mr=#1}{#2}
}
\providecommand{\href}[2]{#2}

\end{document}